\documentclass[a4paper,reqno,final]{amsart}
\usepackage[eso-foot]{svninfo}
\svnInfo $Id: general_wnaf.tex 2106 2012-05-18 09:58:03Z krenn@MATH.TU-GRAZ.AC.AT $

\usepackage{a4wide}
\usepackage{amssymb}
\usepackage[english]{babel}
\usepackage[T1]{fontenc}
\usepackage{url}

\usepackage{xypic}

\usepackage{ifpdf}  
\ifpdf
\usepackage[%
pdftex=true,      
colorlinks=false,   
bookmarks=true,          
bookmarksopen=true,     
bookmarksnumbered=false, 
pdfpagemode=UseOutlines,         
final
]{hyperref}  
\else  
\usepackage[%
dvips,           
colorlinks=false 
]{hyperref}
\fi

\newcommand{\bfeta}{\boldsymbol{\eta}}
\newcommand{\bfxi}{\boldsymbol{\xi}}
\newcommand{\C}{\mathbb{C}}
\newcommand{\calA}{\mathcal{A}}
\newcommand{\calD}{\mathcal{D}}
\newcommand{\order}{\mathfrak{O}}
\renewcommand{\MR}[1]{}
\newcommand{\Q}{\mathbb{Q}}
\newcommand{\R}{\mathbb{R}}
\newcommand{\val}{\mathsf{value}}
\newcommand{\wNADS}{$w$-NADS}
\newcommand{\wNAF}{$w$-NAF}
\newcommand{\Z}{\mathbb{Z}}
\newcommand{\N}{\mathbb{N}}

\makeatletter
\newcommand\undisp[1]{\bgroup\@displayfalse #1\egroup}
\makeatother
\newcommand{\tpmod}[1]{\ensuremath{\undisp{\pmod{#1}}}}

\DeclareMathOperator{\diag}{diag}
\DeclareMathOperator{\id}{id}
\DeclareMathOperator{\interior}{int}

\newtheorem{lemma}{Lemma}[section]
\newtheorem{theorem}{Theorem}
\newtheorem*{unnumbered_theorem}{Theorem}

\newtheorem{proposition}[lemma]{Proposition}
\newtheorem{corollary}[lemma]{Corollary}
\theoremstyle{definition}
\newtheorem{definition}[lemma]{Definition}
\theoremstyle{remark}
\newtheorem{remark}[lemma]{Remark}
\newtheorem{example}[lemma]{Example}

\begin{document}


\title[Existence and Optimality of $w$-Non-adjacent Forms]{Existence and
  Optimality of $w$-Non-adjacent Forms
  \\ with an Algebraic Integer Base}


\author{Clemens Heuberger}

\address{\parbox{12cm}{%
    Clemens Heuberger \\
    Institute of Mathematics \\
    Alpen-Adria-Universit\"at\\
    Universit\"atsstra\ss e 65--67, A-9020 Klagenfurt am W\"orthersee, Austria\\ }}

\email{\href{mailto:clemens.heuberger@aau.at}{clemens.heuberger@aau.at}}


\author{Daniel Krenn}

\address{\parbox{12cm}{%
    Daniel Krenn \\
    Institute of Optimisation and Discrete Mathematics (Math B) \\
    Graz University of Technology \\
    Steyrergasse 30/II, A-8010 Graz, Austria \\}} 

\email{\href{mailto:math@danielkrenn.at}{math@danielkrenn.at} \textit{or}
  \href{mailto:krenn@math.tugraz.at}{krenn@math.tugraz.at}}




\thanks{The authors are supported by the Austrian Science Fund (FWF): S9606,
  that is part of the Austrian National Research Network ``Analytic
  Combinatorics and Probabilistic Number Theory'', and by the Austrian Science
  Fund (FWF): W1230, Doctoral Program
  ``Discrete Mathematics''.}

\thanks{Clemens Heuberger is also supported by the  Austrian Exchange Service
  \"OAD, project number HU 04/2010}


\begin{abstract}
  We consider digital expansions in lattices with endomorphisms acting as
  base. We focus on the $w$-non-adjacent form ($w$-NAF),
  where each block of $w$ consecutive digits contains at most one non-zero
  digit. We prove that for sufficiently large $w$ and an expanding
  endomorphism, there is a suitable digit set  such that
  each lattice element has an expansion as a $w$-NAF. 

  If the eigenvalues of the endomorphism are large enough and $w$ is
  sufficiently large, then the $w$-NAF is shown to minimise the weight among
  all possible expansions of the same lattice element using the same digit
  system.
\end{abstract}


\maketitle


\section{Introduction}
\label{sec:intro}

One main operation in hyperelliptic curve cryptography is the computation of
multiples of a point on a hyperelliptic curve over a finite field. Clearly, we
want to perform that scalar multiplication as efficiently as possible. A
standard method are double-and-add algorithms. But if the hyperelliptic curve
is defined over a field with $q$ elements and we are working in the point group
over an extension (i.e., working over a field with $q^m$ elements), then we can
use a Frobenius-and-add method instead. There the (expensive) doublings are
replaced by the (cheap) evaluations of the $q$-Frobenius endomorphism in the
point group.

In the endomorphism ring of the point group, the Frobenius endomorphism
$\varphi$ acting on the group has a characteristic polynomial $f \in
\Z[X]$. Let $\tau$ be a complex zero of $f$. If we write a $z\in\Z[\tau]$ as
$z=\sum_{j=0}^{\ell-1} \eta_j\tau^j$ for some $\eta_j$ out of a digit set
$\calD$, then we can calculate $zP$ for a point $P$ on the curve by evaluating
$\sum_{j=0}^{\ell-1} \eta_j\varphi^j(P)$. Note that when $z$ is a rational
integer, we are calculating multiples of the point $P$ as mentioned at the
beginning of this section. Therefore we have to understand numeral systems with
an algebraic integer $\tau$ as base.

The sums in the previous paragraph are usually evaluated by a Horner
scheme. There the number of additions when calculating $zP$ corresponds to the
number of non-zero digits (Hamming weight) of our expansion of $z$. Therefore
we are interested in expansions of small weight.  Let $w$ be a positive
integer. An expansion which gives a low Hamming weight is the $w$-non-adjacent
form, $\calD$-\wNAF{} for short, cf.\ \cite{Muir-Stinson:2005:alter-digit-long,
  avanzi:mywnaf, Solinas:2000:effic-koblit}. It is defined by the syntactic
requirement that every block of $w$ consecutive digits contains at most one
non-zero digit.  Suitable conditions on $\calD$ are required such that it is a
$w$-non-adjacent digit set (\wNADS{} for short), which means that each element
of $\Z[\tau]$ has a representation as a $\calD$-\wNAF{}.

In the present paper we give positive results on that existence question. Our
set-up is more general: In Section~\ref{sec:basics}, which contains the
definitions and some basic results, we work in an Abelian group and the base is
represented by an injective endomorphism on that group. In the remaining
article, starting with Section~\ref{sec:lattices}, the set-up is a lattice
$\Lambda$ in $\R^n$ and an injective endomorphism on $\Lambda$ as base.  The
case of algebraic integer bases is a special case of this set-up, cf.\
Examples~\ref{example:number-field-case}
and~\ref{example:number-field-case-and-lattice}.

In Section~\ref{sec:lattices} we prove a necessary condition to be a \wNADS{},
namely that the endomorphism has to be expanding. Section~\ref{sec:tilings}
deals with the setting when the digit set comes from a tiling of
$\R^n$. Theorem~\ref{thm:tiling-wNADS} states that we have a \wNADS{} if $w$ is
sufficiently large. The bound in that result is explicit. Another result of
that kind is given in Section~\ref{sec:min-norm}, generalising a result of
Germ\'an and Kov\'acs \cite{German-Kovacs:2007:number-system-const} to
$\calD$-\wNAF{}s. There minimal norm digit sets are studied. Again we get a
\wNADS{} if $w$ is larger than a constant, which depends (only) on the
eigenvalues of $\Phi$, cf.\ Theorem~\ref{theorem:norm-2}. As an important
example, we discuss the setting of bases $\tau$ coming from hyperelliptic
curves, see above, in Example~\ref{example:alg-curves}.

The last section is devoted to the question of minimality: Are the
$\calD$-\wNAF{}-expansions optimal, i.e., does the $\calD$-\wNAF{}-expansion of
an element minimise the weight among all possible expansions of that element
with the same digit set?  We provide a positive answer for sufficiently large
$w$ and sufficiently large eigenvalues of $\Phi$ in
Theorem~\ref{thm:optimality}.


\section{$w$-Non-Adjacent Forms and Digit Sets}
\label{sec:basics}

In this section, we formally introduce the notion of $w$-non-adjacent forms and
$w$-non-adjacent digits sets.

We consider an Abelian group $\calA$, an injective endomorphism $\Phi$ of
$\calA$ and an integer $w\ge 1$. Let $\calD^\bullet$ be a system of
representatives of those residue classes of $\calA$ modulo $\Phi^w(\calA)$
which are not contained in $\Phi(\calA)$. We set
$\calD=\calD^\bullet\cup\{0\}$.

We call the triple $(\calA, \Phi, \calD)$ a \emph{pre-$w$-non-adjacent digit
  set} (\emph{pre-\wNADS}).

\begin{definition}
  \begin{enumerate}
  \item A word $\bfeta=\eta_{\ell-1}\ldots\eta_0$ over the alphabet $\calD$ is
    said to be a \emph{$\calD$-$w$-non-adjacent form} (\emph{$\calD$-\wNAF}),
    if every factor $\eta_{j+w-1}\ldots\eta_j$, $0\le j\le \ell-w$, contains at
    most one non-zero letter $\eta_k$. Its \emph{value} is defined to be
    \begin{equation*}
      \val(\eta_{\ell-1}\ldots\eta_0)=\sum_{j=0}^{\ell-1} \Phi^j(\eta_j).
    \end{equation*}
    We say that $\bfeta$ is a \emph{$\calD$-\wNAF{} of $\alpha\in\calA$} if
    $\val(\bfeta)=\alpha$.
  \item We say that $\calD$ is a \emph{$w$-non-adjacent digit set}
    (\emph{\wNADS}), if every $\alpha\in\calA$ admits a $\calD$-\wNAF{}.
  \end{enumerate}
\end{definition}

\begin{example}\label{example:number-field-case}
  Let $K$ be a number field of degree $n$, $\order$ be an order in $K$ and
  $\tau\in\order$.  We consider the endomorphism
  $\Phi_{\tau}\colon\order\to\order$ with $\alpha\mapsto \tau\alpha$, i.e.,
  multiplication by $\tau$. Then let $\calD^\bullet$ be a system of
  representatives of those residue classes of $\order$ modulo $\tau^w$ which
  are not divisible by $\tau$ and $\calD=\calD^\bullet\cup \{0\}$. Then
  $(\order, \Phi_\tau,\calD)$ is a pre-\wNADS. Note that
  \begin{equation*}
    \val(\eta_{\ell-1}\ldots\eta_0)=\sum_{j=0}^{\ell-1}\eta_j\tau^j
  \end{equation*}
  for a word $\eta_{\ell-1}\ldots\eta_0$ over the alphabet $\calD$.
\end{example}

We state a few special cases.

\begin{example}\label{example:rational-base}
  Let $\tau\in \Z$, $|\tau|\ge 2$ and $w\ge 1$ be an integer. Consider
  \begin{equation*}
    \calD^\bullet=\left\{ d\in \Z: -\frac{|\tau|^w}{2}<d\le\frac{|\tau|^w}{2},
    \tau\nmid d\right\}
  \end{equation*}
  and $\calD=\calD^\bullet\cup\{0\}$. Then $(\Z, \Phi_\tau, \calD)$ is a
  pre-\wNADS, where $\Phi_\tau$ still denotes multiplication by $\tau$. It can
  be shown that $(\Z,\Phi_\tau,\calD)$ is a \wNADS. This will also be a
  consequence of Theorem~\ref{theorem:norm-2}.
\end{example}

\begin{example}\label{example:imaginary-quadratic}
  Let $\tau$ be an imaginary quadratic integer and $\calD^\bullet$ a system of
  representatives of those residue classes of $\Z[\tau]$ modulo $\tau^w$ which
  are not divisible by $\tau$ with the property that
  \begin{equation*}
    \text{ if $\alpha\equiv\beta\tpmod{\tau^w}$ and $\alpha\in\calD^\bullet$,
      then $|\alpha|\leq|\beta|$}
  \end{equation*}
  holds for $\alpha$, $\beta\in\Z[\tau]$ which are not divisible by
  $\tau$. This means that $\calD$ contains a representative of minimal absolute
  value of each residue class not divisible by $\tau$. As always, we set
  $\calD=\calD^\bullet\cup\{0\}$.

  Then, for $w\ge 2$, $(\Z[\tau],\Phi_\tau, \calD)$ is a \wNADS{} (cf.\
  Heuberger and Krenn~\cite{Heuberger-Krenn:2010:wnaf-analysis-short}), where
  $\Phi_\tau$ still denotes multiplication by $\tau$.

  For $\tau\in\{(\pm 1\pm \sqrt{-7})/2, (\pm 3\pm\sqrt{-3})/2, 1+\sqrt{-1},
  \sqrt{-2}, (1+\sqrt{-11})/2 \}$, this has been shown by
  Solinas~\cite{Solinas:1997:improved-algorithm,Solinas:2000:effic-koblit} and
  Blake, Murty and Xu~\cite{Blake-Kumar-Xu:2005:effic-algor,
    Blake-Murty-Xu:ta:nonad-radix}, cf.\ also Blake, Murty and
  Xu~\cite{Blake-Murty-Xu:2005:naf} for other digit sets to the bases $(\pm
  1\pm \sqrt{-7})/2$.
\end{example}

At several occurrences, it is useful to consider \emph{equivalent} pre-\wNADS.

\begin{definition}
  The pre-\wNADS{} $(\calA, \Phi, \calD)$ and $(\calA', \Phi', \calD')$ are
  said to be \emph{equivalent}, if there is a group isomorphism
  $Q\colon\calA\to\calA'$ such that the diagram
  \begin{equation*}
    \xymatrix{
      \calA \ar[d]_Q \ar[r]_\Phi& \calA \ar[d]_Q \\
      {\calA'} \ar[r]_{\Phi'}& {\calA'}
    }
  \end{equation*}
  commutes and such that $\calD'=Q(\calD)$.
\end{definition}

It is then clear that the following proposition holds.

\begin{proposition}
  Let $(\calA, \Phi, \calD)$ and $(\calA', \Phi', \calD')$ two equivalent
  pre-\wNADS. Then $\calD$ is a \wNADS{} if and only if $\calD'$ is a \wNADS.
\end{proposition}
\begin{proof}
  Straightforward.
\end{proof}

\begin{example}\label{example:number-field-case-and-lattice}
  We continue Example~\ref{example:number-field-case}, i.e., $K$ is a number
  field, $\order$ an order in $K$, $\tau\in\order$, the endomorphism
  considered is $\Phi_\tau$, the multiplication by $\tau$, and the digit set
  $\calD$ is as in Example~\ref{example:number-field-case}.

  The real embeddings of $K$ are denoted by $\sigma_1$, \ldots, $\sigma_s$; the
  non-real complex embeddings of $K$ are denoted by $\sigma_{s+1}$,
  $\overline{\sigma_{s+1}}$, \ldots, $\sigma_{s+t}$, $\overline{\sigma_{s+t}}$,
  where $\overline{\,\cdot\,}$ denotes complex conjugation and $n=s+2t$. The
  \emph{Minkowski map} $\Sigma\colon K\to \R^n$ maps $\alpha\in K$ to
  \begin{equation*}
    \left(\sigma_1(\alpha),\ldots,\sigma_{s}(\alpha), \Re \sigma_{s+1}(\alpha),
    \Im \sigma_{s+1}(\alpha), \ldots, \Re \sigma_{s+t}(\alpha),
    \Im \sigma_{s+t}(\alpha)\right)\in\R^n.
  \end{equation*}
  We write $\Lambda=\Sigma(\order)$ for the image of $\order$ under
  $\Sigma$. Note that $\Lambda$ is a lattice in $\R^n$. We consider the
  $n\times n$ block diagonal matrix
  \begin{equation*}
    A_\tau:=\diag\left(\sigma_1(\tau), \ldots, \sigma_s(\tau), 
      \begin{pmatrix}
        \Re \sigma_{s+1}(\tau)&-\Im\sigma_{s+1}(\tau)\\
        \Im\sigma_{s+1}(\tau)&\Re \sigma_{s+1}(\tau)
      \end{pmatrix},
      \ldots,
      \begin{pmatrix}
        \Re \sigma_{s+t}(\tau)&-\Im\sigma_{s+t}(\tau)\\
        \Im\sigma_{s+t}(\tau)&\Re \sigma_{s+t}(\tau)
      \end{pmatrix}\right)
  \end{equation*}
  and set $\calD':=\Sigma(\calD)$. Then the pre-\wNADS{}
  $(\order,\Phi_\tau,\calD)$ and $(\Lambda, \Phi'_\tau, \calD')$ are easily
  seen to be equivalent, where $\Phi'_\tau(x):=A_\tau\cdot x$ for $x\in\R^n$.

  Note that if $K$ is an imaginary quadratic number field
  (cf.\ Example~\ref{example:imaginary-quadratic}), this construction merely
  corresponds to a straight-forward identification of $\C$ with $\R^2$.
\end{example}

In order to investigate the \wNADS{} property further, it is convenient to
consider the following two maps.

\begin{definition}
  Let $(\calA,\Phi,\calD)$ be a pre-\wNADS. We define
  \begin{enumerate}
  \item $d\colon\calA\to \calD$ with $d(\alpha)=0$ for $\alpha\in\Phi(\calA)$
    and $d(\alpha)\equiv \alpha\pmod{\Phi^w(\calA)}$ for all other
    $\alpha\in\calA$,
  \item  $T\colon\calA\to \calA$ with $\alpha\mapsto\Phi^{-1}
(\alpha-d(\alpha))$.
  \end{enumerate}
\end{definition}

Note that the map $d$ is well-defined as $\calD^\bullet$ contains exactly one
representative of every residue class of $\calA$ modulo $\Phi^w(\calA)$ which
is not contained in $\Phi(\calA)$.  Furthermore, we have $\alpha\equiv
d(\alpha)\pmod{\Phi(\calA)}$ for all $\alpha\in\calA$. Therefore and by the
injectivity of $\Phi$, the map $T$ is well-defined.  We remark that by
definition, we have $T(0)=0$.

We get the following characterisation, which corresponds to the backwards
division algorithm for computing digital expansions from right (least
significant digit) to left (most significant digit).

\begin{lemma}\label{lemma:w-NAF-backwards-division}
  Let $\alpha\in\calA$. Then $\alpha$ has a $\calD$-\wNAF{}
  $\eta_{\ell-1}\ldots \eta_0$ if and only if $T^\ell(\alpha)=0$. In this case,
  we have $\eta_k=d(T^k(\alpha))$ for $0\le k<\ell$. In particular, the
  $\calD$-\wNAF{} of an $\alpha\in\calA$, if it exists, is unique up to leading
  zeros.
\end{lemma}
\begin{proof}
  Assume that $\eta_{\ell-1}\ldots\eta_0$ is a $\calD$-\wNAF{} of $\alpha$. We
  clearly have $\alpha\equiv \eta_0\pmod{\Phi(\Lambda)}$, so that $\alpha$ is
  an element of $\Phi(\Lambda)$ if and only if $\eta_0=0$. Otherwise, the
  \wNAF{}-condition ensures that $\alpha\equiv
  \eta_0\pmod{\Phi^w(\Lambda)}$. In both cases, we get $d(\alpha)=\eta_0$ and
  therefore
  \begin{equation*}
    T(\alpha)=\val(\eta_{\ell-1}\ldots \eta_1).
  \end{equation*}

  Iterating this process yields $T^k(\alpha)=\val(\eta_{\ell-1}\ldots \eta_k)$
  for $0\le k\le \ell$, where $\eta_{\ell-1}\ldots \eta_k$ is a
  $\calD$-\wNAF. For $k=\ell$, we see that $T^\ell(\alpha)$ is the value of the
  empty word, which is zero by the definition of the empty sum.

  Conversely, we assume that $T^\ell(\alpha)=0$. We note that if $d(\beta)\neq
  0$ for some $\beta\in\calA$, we have $\beta-d(\beta)\equiv
  0\pmod{\Phi^w(\Lambda)}$, which results in
  $T^j(\beta)=\Phi^{-j}(\beta-d(\beta))\equiv 0\pmod{\Phi(\Lambda)}$ and
  $d(T^j(\beta))=0$ for $1\le j\le w-1$. Therefore, the word
  $\bfeta=d(T^{\ell-1}(\alpha))\ldots d(T(\alpha))d(\alpha)$ is a
  $\calD$-\wNAF. Iterating the relation $\beta=\Phi(T(\beta))+d(\beta)$ valid
  for all $\beta\in\calA$, we conclude that
  $\alpha=\Phi^\ell(T^\ell(\alpha))+\val(\bfeta)=\val(\bfeta)$.
\end{proof}

\section{Lattices and $\calD$-\wNAF{}s}
\label{sec:lattices}

We now specialise our investigations to the case that the abstract Abelian
group $\calA$ is replaced by a lattice in $\R^n$, i.e.,
$\calA=\Lambda=w_1\Z\oplus\cdots\oplus w_n\Z$ for linearly independent $w_1$,
\ldots, $w_n\in\R^n$. Further let $\Phi$ be an injective endomorphism of $\R^n$
with $\Phi(\Lambda)\subseteq \Lambda$, $w\geq1$ be an integer, and
$\calD^\bullet$ a system of representatives of those residue classes of
$\Lambda$ modulo $\Phi^w(\Lambda)$ which are not contained in $\Phi(\Lambda)$,
and set $\calD=\calD^\bullet\cup\{0\}$.

The results are still applicable to the case of multiplication by $\tau$ in the
order of a number field, as the purpose of
Example~\ref{example:number-field-case-and-lattice} was to describe it as
equivalent to a lattice $\Lambda\subseteq \R^n$ via the isomorphism $\Sigma$.

The aim of this section is to prove a necessary criterion for a pre-\wNADS{} to
be a \wNADS.

\begin{proposition}\label{proposition:necessary-criterion}
  Let $\calD$ be a \wNADS. Then $\Phi$ is expanding, i.e., $|\lambda|>1$ holds
  for all eigenvalues~$\lambda$ of $\Phi$.
\end{proposition}
\begin{proof}
  \begin{enumerate}
  \item We first consider the case that there is an eigenvalue $\lambda$ of
    $\Phi$ with  $|\lambda|<1$. 

    In a somewhat different wording, this has been led to a contradiction by
    Vince~\cite{Vince:1993:replic}. The idea is the following: After a suitable
    change of variables, the endomorphism $\Phi$ can be represented by a Jordan
    matrix such that the first $k$ coordinates, say, correspond to the
    eigenvalue $\lambda$. Thus the first $k$ coefficients of $\val(\bfeta)$ are
    bounded independently of the word $\bfeta$ over the alphabet $\calD$. Thus
    it is impossible to have a representation of all elements of $\Lambda$.
    This is completely independent of the \wNAF-condition (and gives, in fact,
    a stronger result, as representability by any word over the digit set is
    impossible).

  \item We next consider the case that $|\lambda|\ge 1$ for all eigenvalues
    $\lambda$ of $\Phi$ with equality $|\lambda_0|=1$ for at least one
    eigenvalue $\lambda_0$.

    We again follow Vince~\cite{Vince:1993:replic}, see also Kov\'acs and
    Peth\H{o}~\cite{Kovacs-Petho:1991}, to see that $\lambda_0$ must be a root
    of unity. The idea is that $\lambda_0$ is a unit in
    $\Z[\lambda_0,\overline{\lambda_0}]$, as $\overline{\lambda_0}$ is its
    inverse. Therefore, $\lambda$ has absolute norm $\pm 1$. As we already
    assumed that all its absolute conjugates are at least $1$ in absolute
    value, this implies that all absolute conjugates of $\lambda_0$ 
    lie on the unit circle. Thus $\lambda_0$ is a root of unity. 

    As a consequence, there is some $\ell$ such that $\lambda_0^\ell=1$. In
    other words, $1$ is an eigenvalue of $\Phi^\ell$. After a suitable change
    of coordinates, $\Lambda$ can be assumed to be $\Z^n$ and $\Phi$ can be
    represented by a matrix with integer entries. Let $\alpha$ be an
    eigenvector of $\Phi^\ell$ with eigenvalue~$1$. Multiplying $\alpha$ by a
    suitable integer if necessary,  we can assume that
    $\alpha\in\Z^n=\Lambda$. As $\alpha=\Phi^\ell(\alpha)$, we get
    $\alpha\in\Phi^k(\Lambda)$ for all integers $k\ge 0$, which implies that
    $d(T^k(\alpha))=0$ holds for all $k$. Furthermore, we cannot have
    $T^k(\alpha)=0$ for any $k\ge 0$. Thus, $\alpha$ cannot be represented.
    \qedhere
  \end{enumerate}
\end{proof}

\section{Tiling Based Digit Sets}
\label{sec:tilings}

In this section, we consider a fixed lattice $\Lambda\subseteq \R^n$ and an
expanding endomorphism $\Phi$ of $\R^n$ with
$\Phi(\Lambda)\subseteq\Lambda$. We will discuss digit sets constructed from
tilings.

\begin{definition}
  Let $V$ be a subset of $\R^n$. We say that $V$ tiles $\R^n$ by the lattice
  $\Lambda$, if the following two properties hold:
  \begin{enumerate}
  \item $\bigcup_{z\in \Lambda}(z+V)=\R^n$,
  \item $V\cap (z+V)\subseteq \partial V$ holds for all $z\in \Lambda$ with
    $z\neq 0$.
  \end{enumerate}
\end{definition}

We now assume that $V$ be a subset of $\R^n$ tiling $\R^n$ by $\Lambda$. 

\begin{lemma}\label{lemma:tiling-residue-system}Let $w\ge 1$ and
  \begin{equation*}
    \widetilde{\calD}:=\{ \alpha\in \Lambda: \Phi^{-w}(\alpha)\in V \}.
  \end{equation*}
  Then $\widetilde{\calD}$ contains a complete residue system of $\Lambda$
  modulo $\Phi^w(\Lambda)$.

  Furthermore, if $\alpha$, $\alpha'\in\widetilde{\calD}$ with $\alpha \neq
  \alpha'$ and $\alpha\equiv\alpha'\pmod{\Phi^w(\Lambda)}$, then
  $\Phi^{-w}(\alpha)$, $\Phi^{-w}(\alpha')\in\partial V$.
\end{lemma}
\begin{proof}
  Let $\beta\in\Lambda$. Then there is a $\gamma\in\Lambda$ and a $v\in V$ such
  that $\Phi^{-w}(\beta)=\gamma+v$. Setting $\alpha:=\beta-\Phi^w(\gamma)$,
  this implies that
  \begin{equation*}
    \Phi^{-w}(\alpha)=\Phi^{-w}(\beta)-\gamma=v\in V,
  \end{equation*}
  i.e., $\alpha\in\widetilde{\calD}$ and
  $\beta\equiv\alpha\pmod{\Phi^w(\Lambda)}$.

  Assume now $\alpha$, $\alpha'\in\widetilde{\calD}$ with $\alpha \neq \alpha'$
  and $\alpha\equiv\alpha'\pmod{\Phi^w(\Lambda)}$. We write
  $\alpha'=\alpha+\Phi^w(\gamma)$ for a suitable $\gamma\in\Lambda$. We obtain
  \begin{equation*}
    \Phi^{-w}(\alpha')=\Phi^{-w}(\alpha)+\gamma,
  \end{equation*}
  which implies that $\Phi^{-w}(\alpha')\in\partial V$. Analogously, we get
  $\Phi^{-w}(\alpha)\in\partial V$.
\end{proof}

For an integer $w\ge 1$, we choose a subset $\calD^\bullet$ of
$\widetilde{\calD}$ in such a way that $\calD^\bullet$ contains exactly one
representative of every residue class modulo $\Phi^w(\Lambda)$ which is not
contained in $\Phi(\Lambda)$. We also set $\calD:=\calD^\bullet\cup\{0\}$.

\begin{theorem}\label{thm:tiling-wNADS}
  Let $\|\,\cdot\,\|$ be a vector norm on $\R^n$ such that for the
  corresponding induced operator norm, also denoted by $\|\,\cdot\,\|$, the
  inequality $\|\Phi^{-1}\|<1$ holds. Let $r$ and $R$ be positive reals with
  \begin{equation}\label{eq:V-bounds}
    \{ x\in\R^n : \|x\|\le r\} \subseteq V \subseteq \{x\in\R^n: \|x\|\le R\}.
  \end{equation}
  If $w$ is a positive integer such that
  \begin{equation}\label{eq:tiling-endomorphism-bound}
    \|\Phi^{-1}\|^w< \frac1{1+R/r},
  \end{equation}
  then $\calD$ is a \wNADS.
\end{theorem}
\begin{remark}\label{remark:Euclidean-norm-number-field}
  In the case of expansions in an order of a number field
(Example~\ref{example:number-field-case-and-lattice}), we may take
$\|\,\cdot\,\|$ to be the Euclidean norm $\|\,\cdot\,\|_2$, as the
corresponding operator norm fulfils $\|A_\tau^{-1}\|_2=\max\{1/|\sigma_j(\tau)|:
1\le j\le s+t\}$. In this case, \eqref{eq:tiling-endomorphism-bound} is
equivalent to $|\sigma_j(\tau)|^w>1+R/r$ for all $1\le j\le s+t$.
\end{remark}

\begin{proof}[Proof of Theorem~\ref{thm:tiling-wNADS}]
  Let $\alpha\in\Lambda$. We claim that 
  \begin{equation}\label{eq:backwards-division-iteration-bound}
    \|T^k(\alpha)\|\le \frac{R}{1-\|\Phi^{-1}\|^{w}}+\|\Phi^{-1}\|^{k}\cdot\|\alpha\|
  \end{equation}
  holds for all $k$ with the property that 
  $d(T^{k'}(\alpha))=0$ holds for all non-negative $k'$ with $k-w<k'\le k$.

  For $k=0$, \eqref{eq:backwards-division-iteration-bound} is obviously
  true. We assume that \eqref{eq:backwards-division-iteration-bound} holds for
  some $k$. As an abbreviation, we write $\beta=T^k(\alpha)$ and
  $\eta=d(\beta)$. If $\eta=0$, then we have
  \begin{equation*}
    \|T^{k+1}(\alpha)\|=\|T(\beta)\|=\|\Phi^{-1}(\beta)\|\le\|\Phi^{-1}\|\cdot\|\beta\|\le
    \frac{\|\Phi^{-1}\|\cdot R}{1-\|\Phi^{-1}\|^w}+\|\Phi^{-1}\|^{k+1}\cdot \|\alpha\|,
  \end{equation*}
  which proves \eqref{eq:backwards-division-iteration-bound} for $k+1$.

  In the case $\eta\neq 0$, we get
  \begin{align*}
    \|T^{k+w}(\alpha)\|&=\|\Phi^{-w}(\beta-\eta)\|\le \|\Phi^{-1}\|^w\cdot\|\beta\|+\|\Phi^{-w}(\eta)\|\\
  &\le \|\Phi^{-1}\|^w\left(\frac{R}{1-\|\Phi^{-1}\|^{w}}+\|\Phi^{-1}\|^{k}\cdot\|\alpha\|\right)+R
  =  \frac{R}{1-\|\Phi^{-1}\|^{w}}+\|\Phi^{-1}\|^{k+w}\cdot\|\alpha\|,
  \end{align*}
  which is \eqref{eq:backwards-division-iteration-bound} for $k+w$.

  By \eqref{eq:tiling-endomorphism-bound} and
  \eqref{eq:backwards-division-iteration-bound}, we can choose a $k_0$ such
  that 
  \begin{equation}\label{eq:backwards-division-embedding-bound}
    \|\Phi^{-w}(T^k(\alpha))\| \le
    \frac{\|\Phi^{-1}\|^w}{1-\|\Phi^{-1}\|^w}R+\|\Phi^{-1}\|^{k+w}\cdot\|\alpha\|< r
  \end{equation}
  holds for all $k\ge k_0$. 

  If $T^{k_0}(\alpha)=0$, then $\alpha$ admits a $\calD$-\wNAF{} by
  Lemma~\ref{lemma:w-NAF-backwards-division}. Otherwise, choose $k\ge k_0$
  maximally such that $T^{k_0}(\alpha)\in\Phi^{k-k_0}(\Lambda)$. This is
  possible because $\Phi$ is expanding.
  This results in
  $T^k(\alpha)\notin \Phi(\Lambda)$. Then
  \eqref{eq:backwards-division-embedding-bound} implies that
  \begin{equation*}
    \|\Phi^{-w}(T^k(\alpha))\|<r.
  \end{equation*}
  By \eqref{eq:V-bounds}, we conclude that $\Phi^{-w}(T^k(\alpha))$
  is an element of the interior of $V$. 

  By Lemma~\ref{lemma:tiling-residue-system}, we obtain
  $\Phi^{-w}(T^k(\alpha))\in\calD^\bullet$, hence $d(T^k(\alpha))=T^k(\alpha)$
  and $T^{k+1}(\alpha)=0$. Thus $\alpha$ admits a $\calD$-\wNAF{} by
  Lemma~\ref{lemma:w-NAF-backwards-division}.
\end{proof}

\section{Minimal Norm Digit Set}
\label{sec:min-norm}

In this section, we study a special digit set, the minimal norm digit set. In
the case of an imaginary quadratic integer $\tau$, this notion coincides with
the minimal norm representative digit sets introduced by
Solinas~\cite{Solinas:1997:improved-algorithm,Solinas:2000:effic-koblit}.

Let again $\Lambda$ be a lattice in $\R^n$ and $\Phi$ an expansive endomorphism
of $\R^n$ with $\Phi(\Lambda)\subseteq\Lambda$. Choose a positive integer $w_0$
such that $|\lambda|>2^{1/w_0}$ holds for all eigenvalues $\lambda$ of
$\Phi$. Thus the spectral radius of $\Phi^{-1}$ is less than $1/2^{1/w_0}$. We
choose a vector norm $\|\,\cdot\,\|$ on $\R^n$ such that the induced operator
norm (also denoted by $\|\,\cdot\,\|$) fulfils $\|\Phi^{-1}\|<1/2^{1/w_0}$. As
a consequence, we have $\|\Phi^{-1}\|^w<1/2$ for all $w\ge w_0$. 

Again, in the case of expansions in an order of a number field
(Example~\ref{example:number-field-case-and-lattice}), we may take
$\|\,\cdot\,\|$ to be the Euclidean norm $\|\,\cdot\,\|_2$, cf.\
Remark~\ref{remark:Euclidean-norm-number-field}.

Let $V$ be the Voronoi cell of the origin with respect to the point set
$\Lambda$ and the vector norm $\|\,\cdot\,\|$, i.e.,
\begin{equation*}
  V=\{z\in\R^n : \|z\|\le \|z+\alpha\|\text{ holds for all }\alpha\in\Lambda\}.
\end{equation*}

While $V$ does not necessarily tile $\R^n$ by $\Lambda$ (consider the norm
$\|\,\cdot\,\|_\infty$ and the lattice generated by $(1,0)$ and $(0,10)$ in
$\R^2$), for a given integer $w\ge 1$, we can still select a set
$\calD^\bullet$ of representatives of those residue classes of $\Lambda$
modulo $\Phi^w(\Lambda)$ which are not contained in $\Phi^w(\Lambda)$ such that
\begin{equation*}
  \calD^\bullet\subseteq \{\alpha\in\Lambda: \Phi^{-w}(\alpha)\in V\}.
\end{equation*}
As usual, we also set $\calD:=\calD^\bullet\cup \{0\}$ and call it a
\emph{minimal norm digit set modulo~$\Phi^w$}.

Adapting ideas of Germ\'an and
Kov\'acs~\cite{German-Kovacs:2007:number-system-const} to our setting, we prove
the following theorem.

\begin{theorem}\label{theorem:norm-2} If $w\ge w_0$, then $\calD$ is a
  \wNADS.
\end{theorem}
\begin{proof}
  We set $\widetilde M:=\max\{ \|\eta\|:\eta\in\calD\}$. For
  $\beta\in\Lambda$, we have
  \begin{equation*}
    \|T(\beta)\|=\|\Phi^{-1}(\beta-d(\beta))\|\le \|\Phi^{-1}\|(\|\beta\|+\widetilde{M}).
  \end{equation*}
  Setting 
  \begin{equation*}
    M:=\frac{\|\Phi^{-1}\|}{1-\|\Phi^{-1}\|} \widetilde M,
  \end{equation*}
  we see that
  \begin{alignat*}{2}
    \|T(\beta)\|&<\|\beta\|&\quad\text{ if }\|\beta\|&>M,\\
    \|T(\beta)\|&\le M&\quad\text{ if }\|\beta\|&\le M.
  \end{alignat*}
  As $\Lambda$ is a discrete subset of $\R^n$, we conclude that the sequence
  $(T^k(\alpha))_{k\ge 0}$ is eventually periodic for all $\alpha\in\Lambda$.

  For $\beta\in\Phi(\Lambda)$ with $\beta\neq 0$,  we have 
  \begin{equation*}
    \|T(\beta)\|=\|\Phi^{-1}(\beta)\|\le \|\Phi^{-1}\|\cdot\|\beta\|<\|\beta\|. 
  \end{equation*}
  Consider the set
  \begin{equation*}
    P:=\{\beta\in\Lambda: \beta\notin\Phi(\Lambda) \text{ and
    }(T^k(\beta))_{k\ge 0}\text{ is purely periodic}\}.
  \end{equation*}
  The set $P$ is empty if and only if for each $\alpha\in\Lambda$, there is an
  $\ell$ with $T^\ell(\alpha)=0$, i.e., $\alpha$ admits a
  $\calD$-\wNAF. Therefore, by Lemma~\ref{lemma:w-NAF-backwards-division},
  $P$ is empty if and only if $\calD$ is a \wNADS.

  We therefore assume that $P$ is nonempty. We choose an $\alpha\in P$ such
  that $\|\Phi^{-w}(\alpha)\|\ge \|\Phi^{-w}(\beta)\|$ holds for all $\beta\in
  P$. This is possible, since all elements $\beta$ of $P$ fulfil $\|\beta\|\le
  M$, which implies that $P$ is a finite set.

  Next, we choose $\ell>0$ with $T^{\ell}(\alpha)=\alpha$ and set
  $\eta_k=d(T^k(\alpha))$ for $0\le k\le \ell$. We set
  \begin{equation*}
    N:=\{ 0\le k\le\ell : \eta_k\neq 0\}.
  \end{equation*}
  By the \wNAF-condition, we have $|k-k'|\ge w$ for distinct elements $k$ and
  $k'$ of $N$.

  By definition of $T$, we have
  \begin{equation*}
    \alpha=T^{\ell}(\alpha)=\Phi^{-\ell}\Bigl(\alpha-\sum_{k=0}^{\ell-1}
    \Phi^{k}(\eta_k)\Bigr)=
    \Phi^{-\ell}(\alpha)-\sum_{k=0}^{\ell-1}\Phi^{k-\ell}(\eta_k).
  \end{equation*}
  Applying $\Phi^{-w}$ once more and rearranging yields
  \begin{equation}\label{eq:phi-inverse-alpha-explicit}
    \Phi^{-w}(\alpha)=(\id-\Phi^{-\ell})^{-1}\Bigl(-\sum_{\substack{k=0\\k\in N}}^{\ell-1}\Phi^{k-\ell}(\Phi^{-w}(\eta_k))\Bigr).
  \end{equation}
  Note that we restricted the sum to those $k$ corresponding to non-zero
  digits. 

  We claim that 
  \begin{equation}\label{eq:digit-inequality-cycle}
    \|\Phi^{-w}(\eta_k)\|\le \|\Phi^{-w}(T^k(\alpha))\| \le \|\Phi^{-w}(\alpha)\|
  \end{equation}
  holds for $k\in N$. The first inequality is an immediate consequence of the
  definition of $\calD^\bullet$, as
  $\Phi^{-w}(T^k(\alpha))=\Phi^{-w}(\eta_k)+\gamma$ for a suitable
  $\gamma\in\Lambda$. Here, we used that $\eta_k\neq 0$ implies that
  $T^k(\alpha)\notin\Phi(\Lambda)$. Therefore and as $T^{k+\ell}(\alpha)=T^k(T^{\ell}(\alpha))=T^k(\alpha)$, we also get
  $T^k(\alpha)\in P$. By the choice of $\alpha$, we conclude the second
  inequality in \eqref{eq:digit-inequality-cycle}.

  Taking norms in \eqref{eq:phi-inverse-alpha-explicit} yields
  \begin{equation}\label{eq:phi-inverse-alpha-norm-1}
    \|\Phi^{-w}(\alpha)\|\le
    \frac{\|\Phi^{-w}(\alpha)\|}{1-\|\Phi^{-1}\|^{\ell}}\sum_{\substack{k=0\\k\in
    N}}^{\ell-1}\|\Phi^{-1}\|^{\ell-k}.
  \end{equation}
  As $\ell\in N$, we have
  \begin{equation}\label{eq:phi-inverse-alpha-norm-geometric-series}
    \sum_{\substack{k=0\\k\in
    N}}^{\ell-1}\|\Phi^{-1}\|^{\ell-k}\le \|\Phi^{-1}\|^w+\|\Phi^{-1}\|^{2w}+\cdots+\|\Phi^{-1}\|^{mw}=\|\Phi^{-1}\|^w\frac{1-\|\Phi^{-1}\|^{mw}}{1-\|\Phi^{-1}\|^w},
  \end{equation}
  where $m=\lfloor \ell/w\rfloor$. Combining
  \eqref{eq:phi-inverse-alpha-norm-1} and
  \eqref{eq:phi-inverse-alpha-norm-geometric-series} yields
  \begin{equation*}
    \|\Phi^{-w}(\alpha)\|\le
    \frac{\|\Phi^{-1}\|^w}{1-\|\Phi^{-1}\|^w}\frac{1-\|\Phi^{-1}\|^{mw}}{1-\|\Phi^{-1}\|^{\ell}}\|\Phi^{-w}(\alpha)\|<\|\Phi^{-w}(\alpha)\|,
  \end{equation*}
  as $\|\Phi^{-1}\|^w<1/2$, contradiction.
\end{proof}

We restate this result explicitly for expansion in orders of algebraic number
fields.

\begin{corollary}\label{cor:mnr-digit-set-number-field}
  Let $K$ be an algebraic number field of degree $n$, $\sigma_1$, \ldots,
  $\sigma_s$ the real embeddings and $\sigma_{s+1}$, $\overline{\sigma_{s+1}}$,
  \ldots, $\sigma_{s+t}$, $\overline{\sigma_{s+t}}$ be the non-real complex
  embeddings of $K$. 


  Let $\order$ be an order of $K$ and $\tau\in\order$ such that
  $|\sigma_j(\tau)|>1$ holds for all $j$. Let $w$ be an integer with
  \begin{equation*}
    w> \max\left\{ \frac{\log 2}{\log |\sigma_j(\tau)|}: 1\le j\le s+t\right\}.
  \end{equation*}

  Let $\calD^\bullet$ be a system of representatives of those residue classes
  of $\order$ modulo $\tau^w$ which are not divisible by $\tau$ such that 
  \begin{equation*}
    \text{if $\alpha\equiv\beta\tpmod{\tau^w}$ with $\tau\nmid\alpha$ and
      $\alpha\in \calD$, then
      $\sum_{j=1}^{s+t} a_j 
      \Bigl|\sigma_j\Bigl(\frac{\alpha}{\tau^w}\Bigr)\Bigr|^2
      \leq\sum_{j=1}^{s+t} 
      a_j \Bigl|\sigma_j\Bigl(\frac{\beta}{\tau^w}\Bigr)\Bigr|^2$},
  \end{equation*}
  where $a_j=1$ for $j\in\{1,\dots,s\}$ and $a_j=2$ for
  $j\in\{s+1,\dots,s+t\}$. Then $\calD:=\calD^\bullet\cup\{0\}$ is a
  \wNADS.
\end{corollary}

\begin{example}\label{example:alg-curves}
  Let $C$ be an algebraic curve of genus~$g$ defined over $\mathbb{F}_q$ (a
  field with $q$ elements). The Frobenius endomorphism operates on the Jacobian
  variety of $C$ and satisfies a characteristic polynomial $P\in\Z[T]$ of
  degree~$2g$. Let $\tau$ be a root of $P$. Set $K=\Q[\tau]$ and
  $\order=\Z[\tau]$, and denote the embeddings of $K$ by $\sigma_j$. Using
  Corollary~\ref{cor:mnr-digit-set-number-field}, a minimal norm digit set
  modulo $\tau^w$ is a \wNADS{} if
  \begin{equation*}
    w > \frac{\log 4}{\log q}.
  \end{equation*}
  This is true because of the following reasons: The polynomial $P$ fulfils the
  equation
  \begin{equation*}
    P(T) = T^{2g} L(1/T),
  \end{equation*}
  where $L(T)$ denotes the numerator of the zeta-function of $C$ over
  $\mathbb{F}_q$, cf.\
  Weil~\cite{Weil:1948:var-ab-et-courbes-alg, Weil:1971:courbes-alg-et-var-ab}.
  The Riemann Hypothesis of the Weil Conjectures,
  cf.\ Weil~\cite{Weil:1949}, Dwork~\cite{Dwork:1960} and
  Deligne~\cite{Deligne:1974}, state that all zeros of $L$ have absolute value
  $1/\sqrt{q}$. Therefore $|\sigma_j(\tau)| = \sqrt{q}$, which was to show.
\end{example}

\section{Optimality of $\calD$-\wNAF{}s}
\label{sec:optimality}

In this section, we consider a lattice $\Lambda\subseteq \R^n$ and an expanding
endomorphism $\Phi$ of $\R^n$ with $\Phi(\Lambda)\subseteq\Lambda$. 

\begin{definition}
  Let $\bfeta = \eta_{\ell-1}\dots\eta_0$ be a word over the alphabet
  $\calD$. Its \emph{(Hamming-)weight} is the cardinality of $\{j :
  \eta_j\neq0\}$, i.e., the number of non-zero digits in $\bfeta$. 

  Let $z = \val(\bfeta)$. The expansion $\bfeta$ is said to be \emph{optimal}
  if it minimises the weight among all possible expansions of $z$, i.e., if the
  weight of $\bfeta$ is at most the weight of $\bfxi$ for all words $\bfxi$
  over $\calD$ with $\val(\bfxi) = z$.
\end{definition}

We will show an optimality result for $\calD$-\wNAF{}s in
Theorem~\ref{thm:optimality}, where the digit set comes from a tiling as in
Section~\ref{sec:tilings}. 

\begin{lemma}\label{lem:lim-Phi-m-is-0}
  We have
  \begin{equation*}
    \lim_{m\to\infty} \Phi^m(\Lambda):= \bigcap_{m\in\N_0}
  \Phi^m(\Lambda)= \{0\}.
  \end{equation*}
\end{lemma}

\begin{proof}
  Let $\alpha \in \lim_{m\to\infty} \Phi^m(\Lambda) = \bigcap_{m\in\N_0}
  \Phi^m(\Lambda)$. Then there is a sequence $(\beta_m)_{m\in\N_0}$, all
  $\beta_m\in\Lambda$ and with $\beta_m = \Phi^{-m}(\alpha)$. As $\Phi$ is
  expanding,  we obtain $\beta_m \to 0$ as $m$ tends to infinity. The
  lattice $\Lambda$ is discrete, so $\beta_m=0$ for sufficiently large
  $m$. We conclude that $\alpha=0$.
\end{proof}

Now we define the digit set: We start with a subset $V$  of $\R^n$ tiling $\R^n$ by
$\Lambda$.
For a positive integer $w$ let
\begin{equation*}
  \widetilde{\calD} := 
  \{ \alpha\in \Lambda: \Phi^{-w}(\alpha)\in V \}
\end{equation*}
and
\begin{equation*}
  \widetilde{\calD}_{\mathrm{int}} := 
  \{ \alpha\in \Lambda: \Phi^{-w}(\alpha)\in \interior V \},
\end{equation*}
where $\interior V$ denotes the interior of $V$.  We choose a subset
$\calD^\bullet$ of $\widetilde{\calD}$ in such a way that $\calD^\bullet$
contains exactly one representative of every residue class modulo
$\Phi^w(\Lambda)$ which is not contained in $\Phi(\Lambda)$. We also set
$\calD:=\calD^\bullet\cup\{0\}$. This is the same construction as in
Section~\ref{sec:tilings}.

\begin{lemma}\label{lem:intV-singletons}
  Assume that $V \subseteq \Phi(V)$. Then each element of
  $\widetilde{\calD}_{\mathrm{int}} \setminus \{0\}$ has an expansion of
  weight~$1$.
\end{lemma}

\begin{proof}
  Let $\alpha\in\widetilde{\calD}_{\mathrm{int}} \setminus \{0\}$, and let
  $\beta = \Phi^{-\ell}(\alpha) \in\Lambda$ such that the non-negative integer
  $\ell$ is maximal. Therefore $\beta \not\in \Phi(\Lambda)$. We have that
  $\Phi^{-w}(\beta) = \Phi^{-w-\ell}(\alpha)$ is in the interior of
  $\Phi^{-\ell}(V)$. Using $V \subseteq \Phi(V)$ yields $\Phi^{-w}(\beta) \in
  \interior V$, and therefore, by
  Lemma~\ref{lemma:tiling-residue-system}, $\beta\in\calD^\bullet$. Thus
  $\alpha = \Phi^\ell(\beta)$ has an expansion of weight~$1$.
\end{proof}

\begin{theorem}\label{thm:optimality}
  Assume that $V \subseteq \Phi(V)$, $V=-V$ and that there are a vector norm
  $\|\,\cdot\,\|$ on $\R^n$ and positive reals $r$ and $R$ such that 
  \begin{equation}\label{eq:V-bounds-opt}
    \{ x\in\R^n : \|x\|\le r\} \subseteq V \subseteq \{x\in\R^n: \|x\|\le R\}
  \end{equation}
  and such that the induced operator
  norm (also denoted by $\|\,\cdot\,\|$) fulfils $\|\Phi^{-1}\|<\frac{r}{R}$.

  If $w$ is a positive integer such that
  \begin{equation}\label{eq:w-opt-bound}
    \|\Phi^{-1}\|^w < \frac12 \left( \frac{r}{R} - \|\Phi^{-1}\| \right)
  \end{equation}
  and $\calD$ is a \wNADS{}, then the $\calD$-\wNAF{}-expansion of each element
  of $\Lambda$ is optimal.
\end{theorem}

The proof relies on the following optimality result.

\begin{unnumbered_theorem}[Heuberger and Krenn~\cite{Heuberger-Krenn:2011:wnafs-optimality}]
  If
  \begin{equation}\label{eq:cond-lim-Phi-m-is-0}
    \lim_{m\to\infty} \Phi^m(\Lambda) = \{0\},
  \end{equation}
  and if there are sets $U$ and $S$ such that $\calD \subseteq U$, $-\calD
  \subseteq U$, $U \subseteq \Phi(U)$, all elements in $S\cap\Lambda$ are
  singletons (have expansions of weight~$1$) and if
  \begin{equation*}
    \left( \Phi^{-1}(U) + \Phi^{-w}(U) + \Phi^{-w}(U) \right)
    \cap \Lambda \subseteq S \cup \{0\},
  \end{equation*}
  then every $\calD$-\wNAF{} is optimal.
\end{unnumbered_theorem}

\begin{proof}[Proof of Theorem~\ref{thm:optimality}]
  Condition~(\ref{eq:cond-lim-Phi-m-is-0}) is shown in
  Lemma~\ref{lem:lim-Phi-m-is-0}. For the second condition, we choose
  $U=\Phi^w(V)$ and $S=\Phi^w(\interior V)\setminus\{0\}$, and we show
  \begin{equation*}
    \left( \Phi^{-1}(V) + \Phi^{-w}(V) + \Phi^{-w}(V) \right) 
    \subseteq \interior V.
  \end{equation*}
  Optimality then follows, since each element in $S\cap\Lambda$ has a
  weight~$1$ expansion by Lemma~\ref{lem:intV-singletons}. So let $z$ be an
  element of the left hand side of the inclusion
  above. Using~\eqref{eq:V-bounds-opt} and~(\ref{eq:w-opt-bound}) yields
  \begin{equation*}
    \|z\| \leq \|\Phi^{-1}\| R + 2 \|\Phi^{-1}\|^w R < r,
  \end{equation*}
  therefore $z$ is in the interior of $V$. 
\end{proof}


\bibliography{cheub}
\bibliographystyle{amsplain}

\end{document}